%% file: A_Post_QMMM.tex
\documentclass[a4paper,12pt]{article}

\usepackage{mathrsfs}
\usepackage{graphicx}
\usepackage{amsmath}
\usepackage{amsfonts}
\usepackage{amssymb}
\usepackage{amsthm}
\usepackage{textcomp}
\usepackage{color}
\usepackage{geometry}
\usepackage[normalem]{ulem}
\usepackage{placeins}
\usepackage{bm}
\usepackage{wasysym}
\usepackage{titlesec}
\usepackage{stmaryrd}
\usepackage{float}
\usepackage{subfigure}
\usepackage{pdfsync}
\usepackage{hyperref}
\usepackage{algorithm, algorithmic}
\renewcommand{\algorithmiccomment}[1]{\bgroup\hfill//~#1\egroup}

\newtheorem{theorem}{Theorem}[section]
\newtheorem{lemma}{Lemma}[section]
\newtheorem{remark}{Remark}[section]

\newcommand\R{\mathbb{R}}
\newcommand\Rg{\mathcal{R}}
\newcommand\T{\mathcal{T}}

\newcommand\Z{\mathbb{Z}}

\newcommand\E{\mathcal{E}}

\newcommand\del{\delta}

\newcommand{\<}{\langle}
\renewcommand{\>}{\rangle}

\newcommand{\res}[1]{{\rm R}[{#1}]}
\newcommand{\Mrefine}{\mathcal{M}_{\rm r}}

\definecolor{deeppink}{rgb}{1.0, 0.08, 0.58}

\input{notation}

\titleformat{\subsubsection}[runin]{\bf }{
  \arabic{section}.\arabic{subsection}.\arabic{subsubsection}}{
  0.5em}{}[.]
\title{Adaptive QM/MM Coupling for Crystalline Defects
	\footnote{
	HC was partially supported by Thousand Talents Program for Young Professionals, and the Fundamental Research Funds for the Central Universities of China under grant 2017EYT22. ML was partially supported by NSFC grant 91430106, 11771040. HW was partially supported by NSFC grant 11501389, 11471214. YW and LZ were supported by NSFC grant 11471214, 11871339, 11571314.
	}}
\author{
	Huajie Chen
	\footnote{{\tt chen.huajie@bnu.edu.cn}.
	School of Mathematical Sciences, Beijing Normal University.
	},
	Mingjie Liao
	\footnote{{\tt mliao@xs.ustb.edu.cn}.
	Corresponding author.
	Department of Applied Mathematics and Mechanics, University of Science and Technology Beijing.
	}, 
	Hao Wang
	\footnote{{\tt wangh@scu.edu.cn}.
	School of Mathematics, Sichuan University. 
	}, 
	Yangshuai Wang 
	\footnote{{\tt yswang2016@sjtu.edu.cn}
	Institute of Natural Sciences, School of Mathematical Sciences,
  	and Ministry of Education Key Laboratory of Scientific and Engineering Computing (MOE-LSC),
  	Shanghai Jiao Tong University. 
	},
	and Lei Zhang
	\footnote{{\tt lzhang2012@sjtu.edu.cn}
	Institute of Natural Sciences, School of Mathematical Sciences,
  	and Ministry of Education Key Laboratory of Scientific and Engineering Computing (MOE-LSC),
  	Shanghai Jiao Tong University. 
	}
}
\date{}

\begin{document}
\maketitle

\begin{abstract}
	QM (quantum mechenics) and MM (molecular mechenics) coupling methods are widely 
	used in simulations of crystalline defects. In this paper, we construct a residual based 
	{\it a posteriori} error indicator for QM/MM coupling approximations.
    We prove the reliability of the error indicator (upper bound of the true approximation error) and develop some sampling techniques for its efficient calculation. 
    Based on the error indicator and D\"{o}rfler marking strategy, 
    we design an adaptive QM/MM algorithm for crystalline defects
    and demonstrate the efficiency with some numerical experiments.
\end{abstract}

{\bf keywords.} qm/mm coupling, a posteriori error estimate, adaptive algorithm, crystal defects


\section{Introduction}
\label{sec:introduction}

Quantum mechanics and molecular mechanics (QM/MM) coupling methods have been widely used for simulations of large systems in materials science and biology \cite{bernstein09, csanyi04, gao02,	kermode08, ogata01, warshel76, zhang12}.  A QM model is required to accurately treat 
bond breaking/formation, charge transfer, electron excitation and other electronic processes. 
However, the QM calculations can only be applied to systems with hundreds/thousands of
atoms due to their demanding computational cost. By contrast, MM methods based on empirical inter-atomic potentials are able to treat millions of atoms or more, but with reduced accuracy and transferablity (MM can be very accurate at reference configurations or near equilibrium, but may have significant error for general configurations). 
QM/MM coupling methods promise (near-)QM accuracy at (near-)MM computational cost for large-scale atomistic simulations.

In QM/MM simulations, the computational domain is partitioned into QM and MM regions.
The region of primary interest is described by a QM model, and the QM region is embedded in an ambient environment (e.g., bulk crystal) that is described by an MM model.
Some coupling/embedding schemes are applied to link the QM and MM regions.
A natural and fundamental question is how to assign each atom (site) to QM or MM subsystems, in order to achieve the optimal balance between accuracy and computational cost. 
Even for static problems this is not straightforward: we should include the active sites in the QM region when the region of interest is fairly localized and relatively well separated from the environment, however, how to find an optimal partition such that the computational cost can be optimized without loss of accuracy remains unclear. For dynamic problems, this could be more challenging since some sites need to be reassigned as the environments evolve (see, e.g. \cite{csanyi04,duster17,kermode08}). 

The goal of the adaptive QM/MM method is to offer the capability of automatic partition of QM/MM subsystems on the fly according to the error distribution in the process of a simulation. 
This is a distinct advantage over conventional QM/MM method, where a static partition is prescribed for the QM and MM subsystems. 
The adaptive QM/MM method has been proposed in some applications, including the study of important molecular fragments in macromolecules, monitoring molecules entering/leaving binding sites, and tracking proton transfer via the Grotthuss mechanism (see \cite{duster17} and references therein). 
Because the size of the QM region can be set as small as possible (up to the accuracy requirement) in the adaptive QM/MM method, the computational costs can be controlled. Small QM subsystems also facilitate the utilization of high-level 
QM theory and make simulations on long time scales feasible, which may potentially lead to new insights on physical systems.

The efficiency of an adaptive algorithm is determined by the accuracy of {\it a posteriori} error indicator, which indicates the (QM/MM) classification criteria of the atomic sites. 
Despite various existing implementations of adaptive QM/MM coupling methods which mostly rely on empirical error indicators \cite{kerdcharoen1996, kerdcharoen2002, heyden2007, watanabe2014, waller2014, boereboom2016}, up to our best knowledge, we have not seen any rigorous {\it  a posteriori} error estimate for QM/MM coupling.
In fact, recent developments in a similar field, atomistic/continuum coupling methods for crystalline defects
(see, e.g. \cite{Abdulle:2013,arndtluskin07c,Ortner:qnl.1d,OrtnerWang:2014,prud06,Shenoy:1999a,Wang:2017, Liao2018}) have provided valuable insights also on the study of QM/MM methods.

The purpose of this paper is to construct a rigorously justifiable {\it a posteriori} error indicator that is an upper bound of the true error ({\it reliability}), and further design an adaptive QM/MM algorithm.
In this work, we use a prototypical QM/MM model as a proof of concept, with tight binding model as the QM model, and focus only on the static problems. 
We will investigate the adaptive QM/MM coupling with more realistic QM models such as density function theory (DFT) models and  study the dynamic problems in our future work.

\subsubsection*{Outline}

In Section \ref{sec:pre} we brifely describe the tight binding model and QM/MM coupling methods for crystalline defects. 
In Section \ref{sec:analysis}, we derive a residual based {\it a posteriori} error indicator for QM/MM coupling, prove its reliability, and further provide some sampling strategy to accelerate the evaluation of the error indicator.
In Section \ref{sec:adaptive}, we propose an adaptive QM/MM algorithm that automatically adjust the QM and MM regions on the fly according to the proposed {\it a posteriori} error indicator.
In Section \ref{sec:numerics}, we present several numerical experiments for point defects in two dimensional triangular lattice. 
In Section \ref{sec:conclusion}, we make concluding remarks and point out some promising directions for future work.

\subsubsection*{Notation}

We use the symbol $\langle\cdot,\cdot\rangle$ to denote an abstract duality
pairing between a Banach space and its dual space. The symbol $|\cdot|$ normally
denotes the Euclidean or Frobenius norm, while $\|\cdot\|$ denotes an operator
norm.
For the sake of brevity of notation, we will denote $A\backslash\{a\}$ by
$A\backslash a$, and $\{b-a~\vert ~b\in A\}$ by $A-a$.
For $E \in C^2(X)$, the first and second variations are denoted by
$\<\delta E(u), v\>$ and $\<\delta^2 E(u) v, w\>$ for $u,v,w\in X$.

For a finite set $A$, we will use $\#A$ to denote the cardinality of $A$.

The symbol $C$ denotes generic positive constant that may change from one line
of an estimate to the next. When estimating rates of decay or convergence, $C$
will always remain independent of the system size, the configuration of the lattice and the the test functions. The dependence of $C$ will be normally clear from the context or stated explicitly.

\section{Model set up}
\label{sec:pre}
\setcounter{equation}{0}

\subsection{The tight binding model and its site energy}
\label{sec:tb}

\def\Rc{R_{\rm cut}}
\def\Nn{N}

In this paper, we use the tight binding model as the quantum mechanical model, which is a ``minimalist" electronic structure model.
For simplicity of presentation, we consider a `two-centre' tight binding model \cite{goringe97,Papaconstantopoulos15} with a single orbital per atom and the identity overlap matrix.
All results in this paper can be extended directly to general non-self-consistent tight binding models, as described in~\cite[\S~2 and Appendix A]{chen15a}.

Consider a many-particle system consisting of $\Nn$ atoms.
Let $d\in\{2,3\}$ be the space dimension and $\Omega\subset\R^d$ be an {\it index set}  (or  {\it reference configuration}), with $\#\Omega=\Nn$.
An atomic configuration is a map $y : \Omega\to\R^d$ satisfying
\begin{equation} \label{eq:non-interpenetration}
|y(\ell)-y(k)| \geq \mathfrak{m}|\ell-k| \qquad\forall~\ell,k\in\Omega
\end{equation}
with {\em accumulation parameter} $\mathfrak{m} > 0$. In the following, we use $r_{\ell k}:=|y(\ell)-y(k)|$ for brevity of notation.
The `two-centre' tight binding model is formulated in terms of a discrete Hamiltonian, with the matrix elements
\begin{eqnarray}\label{tb-H-elements}
\Big(\mathcal{H}(y)\Big)_{\ell k}
=\left\{ \begin{array}{ll} 
h_{\rm ons}\left(\sum_{j\neq \ell}
\varrho\big(|y({\ell})-y(j)|\big)\right)
& {\rm if}~\ell=k \\[1ex]
h_{\rm hop}\big(|y(\ell)-y(k)|\big) & {\rm if}~\ell\neq k,
\end{array} \right.
\end{eqnarray}
where $h_{\rm ons} \in C^{\mathfrak{n}}([0, \infty))$ is the on-site term,
$\varrho \in C^{\mathfrak{n}}([0, \infty))$ represents the charge density
with $\varrho(r) = 0~\forall r\in[\Rc,\infty)$ and $\Rc>0$ stands for the cutoff radius, 
$h_{\rm hop} \in C^{\mathfrak{n}}([0, \infty))$ is the hopping term with
$h_{\rm hop}(r)=0~\forall r\in[\Rc,\infty)$.
Throughout this paper, we will assume that $\mathfrak{n}\geq 4$.

With the above tight binding Hamiltonian $\mathcal{H}$,
we can define the band energy of the system
\begin{eqnarray}\label{e-band}
E^\Omega(y)=\sum_{s=1}^N f(\varepsilon_s)\varepsilon_s,
\end{eqnarray}
where $(\varepsilon_s)_{s = 1}^N$ are the eigenvalues of $\mathcal{H}(y)$ with associated eigenvectors $\psi_s$ such that
\begin{eqnarray}\label{eigen-H}
\mathcal{H}(y)\psi_s = \varepsilon_s\psi_s\quad s=1,2,\cdots,N,
\end{eqnarray}
and $f$ is the Fermi-Dirac distribution function for the energy states of a system consisting of particles that obey the Pauli exclusion principle,
\begin{eqnarray}\label{fermi-dirac}
f(\varepsilon) = \left( 1+e^{(\varepsilon-\mu)/(k_{\rm B}T)} \right)^{-1}
\end{eqnarray}
with $\mu$ a fixed chemical potential, $k_{\rm B}$ the Boltzmann constant , and $T>0$ the temperature of the system.
We note that it is reasonable to fix the chemical potential $\mu$ in the thermodynamic limit of the
grand canonical ensemble of the electrons \cite{chen16}.

Following \cite{finnis03}, we can distribute the energy to each atomic site
\begin{eqnarray}\label{E-El}
E^\Omega(y)=\sum_{\ell\in\Omega} E_{\ell}^{\Omega}(y)
\qquad{\rm with}\qquad
E_{\ell}^\Omega(y) := \sum_{s}f(\varepsilon_s)\varepsilon_s
\left|[\psi_s]_{\ell}\right|^2,
\end{eqnarray}
which formally defines a site energy $E_{\ell}^{\Omega}(y)$. For the purpose of molecular modeling, we need to justify the {\it regularity and locality}, the {\it isometry and permutation invariance}, and the existence of {\it thermodynamic limit} for this site energy. 

Suppose $\L$ is a countable index set or reference configuration, and $\Omega\subset\Lambda$ is a finite subset.
We denote by $E_\ell^\Omega$ the site energy with respect to the subsystem $\Omega \subset \Lambda$. 
For a domain $A \subset \R^d$, we use the short-hand $E_\ell^{A} := E_\ell^{A \cap \Lambda}$. 

In the tight binding Hamiltonian \eqref{tb-H-elements}, the interaction range of each atom is uniformly localized, which satisfies the assumptions on Hamiltonian matrix elements in \cite{chen15a} (the interactions decays exponentially). 
Then the following lemma from \cite[Theorem 3.1 (i)]{chen15a} implies the existence of the {\it thermodynamic} limit of $E_\ell^{\Omega}$ as $\Omega \uparrow \Lambda$, and guarantees that $E_{\ell}^{\Omega}$ defined in \eqref{E-El} can be taken as a proper (approximate) site energy. 

\begin{lemma}\label{lemma-thermodynamic-limit}
If $y:\L\rightarrow\R^d$ is a configuration satisfying \eqref{eq:non-interpenetration}, then,
	\begin{itemize}
		\item[(i)] {\rm (regularity and locality of the site energy)}
		$E^{\Omega}_{\ell}(y)$ possesses $j$th order partial derivatives with
		$1 \leq j \leq \mathfrak{n}-1$, and there exist positive constants $C_j$ and $\eta_j$ such that
		\begin{eqnarray}\label{site-locality-tdl}
		\left|\frac{\partial^j E^{\Omega}_{\ell}(y)}{\partial [y(m_1)]_{i_1}
			\cdots\partial [y(m_j)]_{i_j}}\right|
		\leq C_j e^{-\eta_j\sum_{l=1}^j|y(\ell)-y(m_l)|}
		\end{eqnarray}
		with $m_k\in\Omega$ and $1\leq i_k\leq d$ for any $1\leq k\leq j$;
		
		\item[(ii)] {\rm (isometry and permutation invariance)}
		If $g:\R^d\rightarrow\R^d$ is an isometry, then
		$E^{\Omega}_{\ell}(y) = E^{\Omega}_{\ell}(g(y))$;
		If  $\mathcal{G}:\Omega\rightarrow\Omega$ is a permutation, then
		$E^{\Omega}_{\ell}(y) = E^{\mathcal{G}^{-1}(\Omega)}_{\mathcal{G}^{-1}
			(\ell)}(y\circ\mathcal{G})$;
		
		\item[(iii)] {\rm (thermodynamic limit)}
		$\displaystyle E_{\ell}(y):=\lim_{R\rightarrow\infty} E^{B_R(\ell)}_{\ell}(y)$
		exists and satisfies (i), (ii).
	\end{itemize}
\end{lemma}

For a finite subset $\Omega \subset \Lambda$, we define the (negative) force
\begin{eqnarray}\label{eq:force}
f^{\Omega}(y):=-\nabla E^{\Omega}(y),
\quad\text{and in component notation,} \quad
\big[f_{\ell}^{\Omega}(y)\big]_i =
-\frac{\partial E^{\Omega}(y)}{\partial [y(\ell)]_i}
\quad 1\leq i\leq d.
\quad
\end{eqnarray}
Using \eqref{E-El}, we have
\begin{eqnarray}\label{Fl-El}
\big[ f_{\ell}^{\Omega}(y) \big]_i = -\sum_{k\in\Omega}
\frac{\partial E_k^{\Omega}(y)}{\partial [y(\ell)]_i},
\end{eqnarray}
which, together with Lemma \ref{lemma-thermodynamic-limit}, yields the
thermodynamic limit of the force $f_{\ell}(y)$, as well as its regularity, locality, and isometry/permutation invariance.

\subsection{Variational formulation for crystalline defects}
\label{sec:defects}

\def\Rdef{R^{\rm def}}
\def\Rg{\mathcal{R}}
\def\Rgnn{\mathcal{N}}
\def\rcut{R_{\rm c}}
\def\Lhom{\L^{\rm hom}}
\def\Ddef{D^{\rm def}}
\def\Ldef{\L^{\rm def}}
\def\Rcore{R_{\rm DEF}}
\def\Adm{{\rm Adm}}
\def\E{\mathcal{E}}
\def\L{\Lambda}
\def\UsH{{\mathscr{U}}^{1,2}}
\def\Usz{{\mathscr{U}^{\rm c}}}
\def\DD{{\sf D}}
\def\ee{{\sf e}}

A rigorous framework for modelling the geometry equilibration of crystalline defects has been developed in 
\cite{chenpre_vardef,2013-defects}, which formulates the equilibration of crystal defects as a variational problem in a discrete energy space, and establishes qualitatively sharp far-field decay estimates for the corresponding equilibrium configuration.
We emphasize that these results rely heavily on a ``locality" assumption of the models, which has been shown for tight binding model in Lemma \ref{lemma:regularity}. 
For sake of simplicity, we only present results on point defects here. All analysis and algorithms can be generated to straight dislocations (see \cite{chenpre_vardef,ehrlacher13}).

Given $d \in \{1, 2, 3\}$, $\mA \in \R^{d \times d}$ non-singular,  $\Lhom := \mA \Z^d$ is the 
homogeneous reference lattice which represents a perfect single lattice crystal formed by identical atoms 
and possessing no defects. 
$\L\subset \R^d$ is the reference lattice with some local defects. The mismatch between $\L$ and $\Lhom$ 
represents possible defects, which are contained in some localized defect cores.
The generalization to multiple defects is straightforward.
For simplicity, we assume the defects are near the origin and 
\begin{eqnarray}\label{ass:ref_config}
\L \setminus B_{\Rcore} = (A \Z^d) \setminus B_{\Rcore}
\end{eqnarray}
with $\Rcore\geq 0$.
For analytical purposes, we assume that there exits a regular partition $\mathcal{T}_{\Lambda}$ of $\R^d$ 
into triangles if $d=2$ and tetrahedra if $d=3$, whose nodes are the reference sites $\Lambda$.

Recall that the deformed configuration of the infinite lattice $\L$ is a map $y: \L\rightarrow\R^d$, which can be decomposed as
\begin{eqnarray}\label{y-u}
y(\ell) = \ell + u(\ell)  \qquad\forall~\ell\in\Lambda
\end{eqnarray}
with $u:\Lambda\rightarrow\mathbb{R}^d$ the displacement with respect to the reference configuration $\L$.

If $\ell\in\Lambda$ and $\ell+\rho\in\Lambda$, then we define the finite difference
$D_\rho u(\ell) := u(\ell+\rho) - u(\ell)$. For a subset $\Rg \subset \Lambda-\ell$, we
define $D_\Rg u(\ell) := (D_\rho u(\ell))_{\rho\in\Rg}$,
and $Du(\ell) := D_{\Lambda-\ell} u(\ell)$.
For $\gamma > 0$ we define the (semi-)norms
\begin{eqnarray*}
	\big|Du(\ell)\big|_\gamma := \bigg( \sum_{\rho \in \L-\ell} e^{-2\gamma|\rho|}
	\big|D_\rho u(\ell)\big|^2 \bigg)^{1/2}
	\quad{\rm and}\quad
	\| Du \|_{\ell^2_\gamma} := \bigg( \sum_{\ell \in \L}
	|Du(\ell)|_\gamma^2 \bigg)^{1/2}.
\end{eqnarray*}
All (semi-)norms $\|\cdot\|_{\ell^2_\gamma}, \gamma > 0,$ are equivalent, see \cite{ortner12} (also \cite[Appendix A]{chen15b}).
We can now define the natural function space of finite-energy displacements,
\begin{displaymath}
\UsH(\L) := \big\{ u : \L \to \R^d, \| Du \|_{\ell^2_\gamma} < \infty \big\}.
\end{displaymath}
We denote $\UsH(\L)$ by $\UsH$ whenever it is clear from the context, .

Let $E_\ell$ denote the site energy we defined in Lemma \ref{lemma-thermodynamic-limit} (iii). 
Due to its translation invariance, we define $V_{\ell} : (\R^d)^{\L-\ell}\rightarrow\R$ by
\begin{eqnarray}
V_{\ell}(Du) := E_{\ell}(x_0+u) \qquad{\rm with}\quad
x_0:\L\rightarrow\R^d ~~{\rm and}~~ x_0(\ell)=\ell~~\forall~\ell\in\L.
\end{eqnarray}
For a displacement $u$ with $x_0+u$ satisfying \eqref{eq:non-interpenetration},  we can formally define the energy-difference functional
\begin{eqnarray}\label{energy-difference}
\mathcal{E}(u) := \sum_{\ell\in\Lambda}\Big(E_{\ell}(x_0+u)-E_{\ell}(x_0)\Big) 
= \sum_{\ell\in\Lambda}\Big(V_{\ell}(Du(\ell))-V_{\ell}(\pmb{0})\Big).
\end{eqnarray}
It was shown in \cite[Theorem 2.7]{chenpre_vardef} (see also \cite{2013-defects}) that, 
if $\del\E(0) \in (\UsH)^*$, then $\E$ is well-defined on the space $\Adm_0$ and in fact
$\E \in C^{\mathfrak{n}-1}(\Adm_0)$, where
\begin{displaymath}
\Adm_{\frak{m}}(\L) := \big\{ u \in \UsH(\L), ~
|x_0(\ell)+u(\ell)-x_0(m)-u(m)| > \frak{m} |\ell-m|
\quad\forall~  \ell, m \in \L
\big\}.
\end{displaymath}
Whenever it is clear from the context, we will denote $\Adm_{\frak{m}}(\L)$ by $\Adm_{\frak{m}}$.
Let $\Adm_0=\cup_{\mathfrak{m}>0}\Adm_{\mathfrak{m}}$.
Due to the decay imposed by the condition $u\in\UsH$, any displacement $u\in\Adm_0$ belongs to $\Adm_{\mathfrak{m}}$ with some constant $\mathfrak{m}>0$.

To this end, we can rigorously formulate the variational problem for the equilibrium state as,
\begin{equation}\label{eq:variational-problem}
\bar u \in \arg\min \big\{ \E(u), u \in \Adm_0 \big\},
\end{equation}
where ``$\arg\min$'' is understood as the set of local minima.
Equivalently, the minimizer $\bar{u}$ satisfies the following first and second order optimality conditions
\begin{eqnarray}
\label{eq:optimality-1}
\b\< \delta\E(\bar{u}) , v\b\> = 0 , \qquad
\b\< \delta^2\E(\bar{u}) v , v\b\> \geq 0,
& \qquad\forall~v\in\UsH .
\end{eqnarray}

Alternatively, we may consider the force equilibrium formulation instead of the energy minimization formulation:
\begin{eqnarray}\label{eq:problem-force}
{\rm Find}~ \bar{u} \in \Adm_0, ~~ {\rm s.t.} \quad
f_{\ell}(\bar{u}) = 0 \qquad \forall~\ell\in\Lambda ,
\end{eqnarray}
where
\begin{eqnarray}\label{eq:force-Du}
f_{\ell}(u) = -\nabla_{\ell} \E(u)   = - \sum_{\rho\in\ell-\L}
V_{\ell-\rho,\rho}\big(Du(\ell-\rho)\big) + \sum_{\rho\in\L-\ell}  V_{\ell,\rho}\big(Du(\ell)\big). \qquad
\end{eqnarray}
Note that any minimizer of \eqref{eq:variational-problem} also solves \eqref{eq:problem-force}.

The second part of \eqref{eq:optimality-1} is usually difficult to justify analytically. 
Hence we impose the following strong stablility condition for the minimizer $\bar{u}$, namely,
\begin{equation}\label{eq:strong-stab}
\exists~ \bar{c} > 0 ~~\text{ s.t. }\quad
\big\< \delta^2 \E(\bar u) v, v\big\> \geq \bar{c}
\| Dv \|_{\ell^2_\gamma}^2	\qquad \forall v \in\UsH.
\end{equation}
The constants $\bar{c}$ has mild dependence on the parameter $\gamma$. 
Nevertheless, since all norms $\|\cdot\|_{\ell^2_{\gamma}}$ are equivalent, we hereafter ignore this dependence.

The following result from \cite{chenpre_vardef} gives the decay estimates for the equilibrium state for point defects.

\begin{lemma}	\label{lemma:regularity}
	Let $\gamma>0$.
	If $\bar u \in \Adm_0$ is a strongly stable solution to \eqref{eq:variational-problem} 
	in the sense that \eqref{eq:strong-stab} is satisfied,
	then there exists a constant $C > 0$ such that
	\begin{align}\label{decay-estimate}
	|D\bar{u}(\ell)|_\gamma \leq C (1+|\ell|)^{-d}. 
	\end{align}
\end{lemma}

\subsection{QM/MM coupling}
\label{sec:qmmm}

\def\Adm{{\rm Adm}}
\def\E{\mathcal{E}}
\def\L{\Lambda}
\def\DD{{\sf D}}
\def\ee{{\sf e}}
\def\LQM{\Lambda^{\rm QM}}
\def\LMM{\Lambda^{\rm MM}}
\def\LFF{\Lambda^{\rm FF}}
\def\Lbuf{\Lambda^{\rm BUF}}
\def\OQM{\Omega^{\rm QM}}
\def\OMM{\Omega^{\rm MM}}
\def\OFF{\Omega^{\rm FF}}
\def\Obuf{\Omega^{\rm BUF}}
\def\RQM{R_{\rm QM}}
\def\RMM{R_{\rm MM}}
\def\RFF{R_{\rm FF}}
\def\Rbuf{R_{\rm BUF}}
\def\VMM{V^{\rm MM}}
\def\Vb{V^{\rcut}_{\#}}
\def\EH{\mathcal{E}^{\rm H}}
\def\uH{\bar{u}^{\rm H}}
\def\wD{\widetilde{D}}
\def\AH{\Adm^{\rm H}_0}
\def\Usx{\mathscr{U}^{\rm H}}
\def\uH{\bar{u}^{\rm H}}

To solve the variational problem \eqref{eq:variational-problem} approximately, we must restrict the infinite dimensional space $\Adm_0$ over $\L$ to a finite dimensional subspace over some bounded domain with  artificial boundary conditions. 
The significant computational cost (roughly speaking, cube of the degrees of freedom) drastically limits the system size which can be handled by the QM models (in this paper, the tight binding model). 
The QM/MM coupling schemes combine the accuracy of QM models with the low computational cost of MM models, and therefore allow simulations with much larger systems.

Generally speaking, QM/MM coupling schemes can be classified according to
whether they link the QM and MM regions on the level of energies or forces \cite{bernstein09,chen15b}: 
the energy-based methods build a hybrid total energy
functional and look for the minimizer of this functional; while the force-based methods solve the force balance equation with QM and MM contributions
and possibly with an interpolation between the two in a transition region. 
We will focus on energy-based methods in this paper, and all our analysis and algorithms can be generalized to force-based methods without too much difficulty. 

The first step of QM/MM algorithm is to decompose the reference configuration $\L$ into three disjoint sets,
$\Lambda = \LQM\cup \LMM\cup \LFF$,  where $\LQM$ denotes the QM region, $\LMM$ denotes the MM region, and $\LFF$ denotes the far-field region where atom positions will be frozen according to the far-field predictor.
Moreover, we define a buffer region $\Lbuf\subset\LMM$ surrounding $\LQM$ such that all atoms in $\Lbuf\cup\LQM$ are involved in the evaluation of the site energies in $\LQM$ using the tight binding model. (see Figure \ref{qmmmgeom} for a schematic plot for the case of a two dimensional point defect)
More precisely, we require
\begin{eqnarray}\label{buf}
B_{\rcut}(\ell) \subset \LQM\cup\Lbuf \qquad\forall~\ell\in\LQM
\end{eqnarray}
with some cutoff distance $\rcut>0$.
Due to the locality in Lemma \ref{lemma-thermodynamic-limit}, 
the error from truncation of the buffered layer $\Lbuf$  decays exponentially fast as $\rcut$ increases. Therefore, $E^{\Lbuf\cup\LQM}_{\ell}$ is a good approximation of $E_{\ell}$ for sufficiently large $\rcut$.
For simple cases, we can use balls centred at the defect core to decompose $\Lambda$, 
and use parameters $\RQM$, $\RMM$ and $\Rbuf(\geq\rcut)$ to represent the respective radii 
(see also Figure \ref{fig:qmmmgd} for a schematic plot).

In the MM region, we approximate the tight binding site potential $V_{\ell}$ by some MM site potential $\VMM(Du(\ell))$,
which will be constructed such that:  
(a) it is cheap to evaluate, usually an explicit function of the atomic configuration;
(b) it only depends on finitely many atoms within a finite range neighbourhood, say, only on sites in $B_{\rcut}(\ell)$;
(c) it is accurate enough when the local atomic configuration is close to perfect lattice.
Note that when $\ell\in \LMM$ is far away from defects, e.g. $\RQM>\Rcore+\rcut$,  the potential $\VMM$ becomes homogeneous and does not depend on $\ell$.
Typically, we can use a Taylor expansion with respect to the reference configuration $x_0$ as follows (see also \cite[eq. (36)]{chen15b}).
Define $\Vb:\big(\R^d\big)^{\mathcal{R}}\rightarrow\R$ as,
\begin{eqnarray*}
\Vb\big(D_{\mathcal{R}}u(\ell)\big) := E_{\ell}^{\L\cap B_{\rcut}(\ell)}(x_0+u)	\quad\forall ~ |\ell|>\Rcore+\rcut
\quad{\rm with}~~\mathcal{R}=B_{\rcut}\cap\big(\Lhom\backslash 0\big) .
\end{eqnarray*}
The MM potential is given by 
\begin{eqnarray}\label{taylor}
\VMM\big({\bm g}\big) 
:= \Vb({\bf 0}) + \sum_{j=1}^k \frac{1}{j!} \delta^j \Vb({\bf 0})\left[{\bm g}^{\otimes j}\right]
\quad{\rm with}~~k\geq 2,
\end{eqnarray}
where $\delta^j \Vb({\bf 0})\left[{\bm g}^{\otimes j}\right]$ denotes the $j$-th order variations,
e.g., $ \delta\Vb({\bf 0})\left[{\bm g}\right] = \langle\delta\Vb({\bf 0}),{\bm g}\rangle$ and
$\delta^2\Vb({\bf 0})\left[{\bm g}^{\otimes 2}\right] = \langle\delta^2\Vb({\bf 0}){\bm g},{\bm g}\rangle$.
This construction is used throughout the numerical experiments in Section \ref{sec:numerics}.

The QM/MM hybrid energy difference functional approximates the QM energy difference functional $\E$ by
\begin{eqnarray}\label{eq:hybrid_energy}
\quad \E^{\rm H}(u) 
= \sum_{\ell\in \LQM} 
\Big( V_{\ell}\big(D u(\ell)\big)  - V_{\ell}\big(\pmb{0}\big) \Big) 
+ \sum_{\ell\in \LMM\cup\LFF} 
\Big( \VMM\big(D u(\ell)\big) - \VMM\big(\pmb{0}\big) \Big) \qquad
\end{eqnarray}
and replace the admissible set $\Adm_0$ by 
\begin{eqnarray}\label{e-mix-space}
\Adm_{0}^{\rm H} := \Adm_{0} \cap \Usx 
\qquad \text{with} \qquad
\Usx := \left\{ u \in \UsH  ~\lvert~ u=0~{\rm in}~\Lambda^{\rm FF} \right\} .
\end{eqnarray}
Finally, the energy-based QM/MM energy coupling scheme, as an approximation of \eqref{eq:variational-problem},  is the following finite dimensional minimization problem.
\begin{eqnarray}\label{problem-e-mix}
\bar{u}^{\rm H} \in \arg\min\big\{ \mathcal{E}^{\rm H}(u) ~\lvert~ u\in \AH \big\},
\end{eqnarray}

Let $\uH$ be the approximate equilibrium state of \eqref{problem-e-mix}. The Taylor expansion construction of the MM site potential about the far-field lattice state (see \cite[\S 4.1]{chen15b})
and the decay estimate in Lemma \ref{lemma:regularity} lead to the convergence of $\uH$ to $\bar u$ and {\it a priori} error estimates with respect to the size of QM 
and MM regions (see \cite[\S 4.2]{chen15b}).
To be more precise, for a two dimensional triangular lattice with point defects, if the MM site potential is given by second order Taylor expansion \eqref{taylor}, 
then we have the following {\it a priori} error estimate for the QM/MM approximation \eqref{problem-e-mix} (a special case of \cite[Theorem 4.1]{chen15b})
\begin{eqnarray}\label{a_priori}
\|\uH-\bar{u}\|_{\UsH} \leq C\Big( \RQM^{-3} + \RMM^{-1} + \exp(-\kappa\rcut) \Big)
\end{eqnarray}
with some constants $C,\kappa>0$ independent of $\RQM$, $\RMM$ and $\rcut$. 
We observe immediately from this estimate that, to balance different contributions to the error and achieve (quasi) optimal computational costs, 
one should take $\RMM\approx\RQM^3$ for sufficiently large $\rcut \approx \log \RMM$.

In our analysis and algorithms, the MM potentials do not need to be restricted to the constructions \eqref{taylor} or those in \cite{chen15b}, 
it suffices to make the following assumption that the QM/MM approximation $\uH$ converges to the exact equilibria in the sense of
\begin{eqnarray}\label{ass:convergence_QMMM}
\lim_{\RQM\rightarrow\infty} \|\uH-\bar{u}\|_{\UsH} = 0.
\end{eqnarray}
This systematic convergence \eqref{ass:convergence_QMMM}  is a basic requirement for a reliable QM/MM scheme.

\section{A posteriori error estimates}
\label{sec:analysis}
\setcounter{equation}{0}

In this section, we derive an {\it a posteriori} error indicator for QM/MM approximations, and show its reliability such that the true error is bounded from above by the error indicator.
Furthermore, we design certain sampling techniques to improve the efficiency of evaluating the indicator in practical calculations.

\subsection{Residual estimates}
\label{sec:res}

For any solution $\uH\in\UsH$ of the QM/MM approximation \eqref{problem-e-mix}, we define the residual $\res{\uH}$ as a functional on $\UsH$:
\begin{eqnarray}\label{eq:res} \nonumber
\res{\uH}(v) := \big\< \delta\E(\uH),v \big\> 
= \sum_{\ell\in\L}\big\< \delta V_{\ell}(D\uH), Dv(\ell) \big\>,
\qquad\forall~v\in\UsH .
\end{eqnarray}
Let $\|\cdot\|_{-1}$ be the dual norm of $\UsH$, the following lemma indicates that $\|\res{\uH}\|_{-1}$ 
provides both lower and upper bounds of the approximation error.

\begin{lemma}\label{lemma:res}
	Let $\bar{u}$ and $\uH$ be the solutions of \eqref{eq:variational-problem} and \eqref{problem-e-mix}, respectively.
	If $\bar{u}$ is storngly stable in the sense of \eqref{eq:strong-stab}
	and $\RQM$ is sufficiently large, then there exist constants $c$ and $C$ such that
	\begin{eqnarray}\label{res-bound}
	c\|\bar{u}-\uH\|_{\UsH} \leq \|\res{\uH}\|_{-1} \leq C\|\bar{u}-\uH\|_{\UsH} .
	\end{eqnarray}
\end{lemma}

\begin{proof}
	Let $r>0$ be such that $B_r(\bar{u}) \subset \Adm_{\frak{m}}$ for some $\frak{m} > 0$.
	Since we have assumed $\mathfrak{n}\geq 4$, it follows from Lemma \ref{lemma-thermodynamic-limit} that $\E\in C^3(\Adm_0)$.	Therefore $\delta\E$ and $\delta^2\E$ are Lipschitz continuous in $B_r(\bar{u})$ with
	uniform Lipschitz constants $L_1$ and $L_2$, i.e., for any $w\in B_r(\bar{u})$
	\begin{align}
	\label{proof-4-1-2}
	\|\delta\E(\bar{u})-\delta\E(w)\|
	&\leq L_1\|D\bar{u}-Dw\|_{\ell^2_\gamma},
	\\[1ex]
	\label{proof-4-1-3}
	\|\delta^2\E(\bar{u})-\delta^2\E(w)\|
	&\leq L_2\|D\bar{u}-Dw\|_{\ell^2_\gamma} .
	\end{align}
	Using \eqref{ass:convergence_QMMM}, we can take $\RQM$ sufficiently large such that $\uH\in B_r(\bar{u})$.
	
	It follows from first order optimality \eqref{eq:optimality-1} and the Lipschitz continuity of $\delta\E$ \eqref{proof-4-1-2} that
	\begin{eqnarray*}
		\res{\uH}(v) = \big\< \delta\E(\uH)-\delta\E(\bar{u}) , v \big\>
		\leq L_1\|D\bar{u}-D\uH\|_{\ell^2_\gamma} \|Dv\|_{\ell^2_\gamma}
		\qquad\forall~v\in\UsH ,
	\end{eqnarray*}
	which leads to the lower bound estimate
	\begin{eqnarray}\label{proof-4-1-4}
	\|\res{\uH}\|_{-1} \leq C\|\bar{u}-\uH\|_{\UsH}
	\end{eqnarray}
	with the constant $C$ depending on $\gamma$ and $L_1$.
	
	For the upper bound estimate, the Lipschitz continuity of $\delta^2\E$ \eqref{proof-4-1-3} and the strong stability condition \eqref{eq:strong-stab} imply the existence of $\tilde{r}\in (0,r)$, such that for any $w\in B_{\tilde{r}}(\bar{u})$
	\begin{eqnarray*}
		\big\< \delta^2 \E(w) v, v\big\> \geq \frac{\bar{c}}{2}
		\| Dv \|_{\ell^2_\gamma}^2	\qquad \forall v \in\UsH.
	\end{eqnarray*}
	Note that \eqref{ass:convergence_QMMM} implies that for $\RQM$ large enough, $\uH\in B_{\tilde{r}}(\bar{u})$.
	Therefore,
	\begin{multline*}
	\quad
	\|\res{\uH}\|_{-1}\|D\bar{u}-D\uH\|_{\ell^2_\gamma}
	~\geq~ \res{\uH}(\bar{u}-\uH) = \big\< \delta\E(\uH)-\delta\E(\bar{u}) , \uH-\bar{u} \big\>
	\\[1ex]
	~=~ \big\< \delta^2\E(w)(\uH-\bar{u}) , \uH-\bar{u} \big\>
	~\geq~ \frac{\bar{c}}{2}\|D\bar{u}-D\uH\|^2_{\ell^2_\gamma},
	\qquad
	\end{multline*}
	where $w=t\bar{u}+(1-t)\uH$ with some $t\in(0,1)$.
	This leads to the estimate
	\begin{eqnarray}\label{proof-4-1-5}
	\|\res{\uH}\|_{-1} \geq c\|\bar{u}-\uH\|_{\UsH}
	\end{eqnarray}
	with some constant $c$ depending on $\gamma$ and $\bar{c}$.
	We complete the proof by combining \eqref{proof-4-1-4} and \eqref{proof-4-1-5}.
\end{proof}

\subsection{A practical a posteriori error indicator}
\label{sec:posteriori}

We observe from Lemma \ref{lemma:res} that an ideal {\it a posteriori} error indicator is 
\begin{eqnarray}\label{eq:ideal_posteriori_indicator}
\eta^{\rm ideal}(\uH) := \|\res{\uH}\|_{-1}.
\end{eqnarray}
The upper and lower bound estimate of the residual dual norm can ensure the reliability and efficiency of the error indicator.
However, the ideal error indicator \eqref{eq:ideal_posteriori_indicator} is not computable and can not be used directly in practice
since $\delta V_{\ell}(\uH)$ is very complicate to compute for a QM model.
The aim of this section is to construct an {\it a postriori} error indicator that can be computed from the QM/MM approximation $\uH$ with moderate 
computational cost and meanwhile, can control the error $\|\uH-\bar{u}\|_{\UsH}$ from above.

A natural idea is to use the force $f_{\ell}(\uH)$ (definded by \eqref{eq:force-Du}) to construct the error indicator, 
since
(a) the equilibrium state $\bar{u}$ satisfies the force balance equation $f_{\ell}(\bar{u})=0~(\forall~\ell\in\L)$, 
and hence $\big|f_{\ell}(\uH)\big|=\big|f_{\ell}(\uH)-f_{\ell}(\bar{u})\big|$ can be related to the error $\|\bar{u}-\uH\|$,
(b) comparing with $\delta V_{\ell}(\uH)$, the QM force $f_{\ell}(\uH)$ is much easier to compute 
by using the Hellmann-Feynman formula (see e.g. \cite{martin04}).

We show in the following theorem an {\it a posteriori} error indicator, which gives an upper bound of the error $\|\bar{u}-\uH\|_{\UsH}$.

\begin{theorem}\label{theorem:upperbound}
	Let $\bar{u}$ and $\uH$ be the solutions to \eqref{eq:variational-problem} and \eqref{problem-e-mix}, respectively.
	If $\bar{u}$ is storngly stable in the sense of \eqref{eq:strong-stab}
	and $\RQM$ is sufficiently large, then there exists a constant $C$ such that
	\begin{eqnarray}\label{eta-upperbound}
	\|\bar{u}-\uH\|_{\UsH} \leq C\eta(\uH) ,
	\end{eqnarray}
where 
\begin{eqnarray}\label{eq:error-indicator-fl}
\eta(\uH) := \left\{
\begin{array}{ll}
\displaystyle
\sum_{\ell\in\L} \log(2+|\ell|)\cdot\b|f_{\ell}(\uH)\b|, \quad & {\rm if} ~ d=2,
\\[1ex]
\displaystyle
\left(\sum_{\ell\in\L} \b|f_{\ell}(\uH)\b|^{\frac65}\right)^{\frac56},  & {\rm if}~ d=3.
\end{array} \right.
\end{eqnarray}
\end{theorem}

\begin{proof}
	For a displacement $v\in\UsH$, we define the equivalence classes
	\begin{eqnarray*}
	[v] := \left\{ v+t ~:~ t \in \R^d \right\} .
	\end{eqnarray*}
	Due to the translation invariance in Lemma \ref{lemma:regularity} (ii), 
	there is no need to make the distinction between $v$ and $[v]$.
	More specifically, we have that for any $w\in[v]$, $Dv(\ell)=Dw(\ell)$
	and $f_{\ell}(v)=f_{\ell}(w),~\forall~\ell\in\L$.
	It follows from \cite[Proposition 12]{ortner12} and \cite[Theorem 2.2]{OrtnerSuli12} that
	for any $v\in\UsH$,
	\begin{eqnarray}
	\label{estimateOS_d2}
	&|v(\ell)-v(0)| \leq C\|v\|_{\UsH}\log(2+|\ell|)    & {\rm if}~d=2  ~~ {\rm and}
	\\[1ex]
	\label{estimateOS_d3}
	& \text{there exists a } v_0\in[v]  \text{ such that } v_0\in \ell^6  & {\rm if} ~d=3. 
	\end{eqnarray}
	
	For $d=2$, \eqref{estimateOS_d2} and the fact $\sum_{\ell\in\L}f_{\ell}(\uH) =0$ 
	imply that, for any $v\in\UsH$,
	\begin{multline*}
	\qquad
	\res{\uH}(v)  =  \<\delta\E(\uH),v\> = \sum_{\ell\in\L} f_{\ell}(\uH)v(\ell)
	= \sum_{\ell\in\L} f_{\ell}(\uH)\tilde{v}(\ell)
	\\
	\leq C \sum_{\ell\in\L} f_{\ell}(\uH)\|\tilde{v}\|_{\UsH}\log(2+|\ell|)  \leq C \eta(\uH)\|v\|_{\UsH},
	\qquad\qquad\qquad
	\end{multline*}
	where $\tilde{v}=v-v(0)\in[v]$.
	This inequality together with \eqref{res-bound} completes the proof of $d=2$ case.
	
	For $d=3$, we can choose $v_0\in[v]$ as in \eqref{estimateOS_d3} to obtain that,
	for any $v\in\UsH$,
	\begin{eqnarray*}
		\res{\uH}(v)  = \sum_{\ell\in\L} f_{\ell}(\uH)v_0(\ell) 
		\leq C \|f_{\ell}(\uH) \|_{\ell^{\frac65}} \|v_0\|_{\ell^6}  \leq C \eta(\uH)\|v\|_{\UsH},
	\end{eqnarray*}
	and use similar argument to complete the proof.
\end{proof}

\begin{remark}
	Although we have stated the {\it a posteriori} error indicators for both $d=2$ and $d=3$ cases in Theorem \ref{theorem:upperbound}, we will focus on the implementations of two dimensional systems in this paper.
	Three dimensional systems will be investigated in our future works.
\end{remark}

\begin{remark}\label{remark:lowerbound}
	By approximating $\|\res{\uH}\|_{-1}$ with the error indicator $\eta(\uH)$, 
	we keep only the upper bound estimate in \eqref{res-bound}, but may have lost the lower bound.
	The design of reliable and efficient error indicator (with both upper and lower bound estimates) for QM/MM schemes may require more involved constructions and analysis, and will be investigated in our future work. See \cite{HW_SY_2018_Efficiency_A_Post_1D} for a recent advance in this direction.
\end{remark}

The error indicator \eqref{eq:error-indicator-fl} is still not computable since 
the sum over $\ell\in\L$ is an infinite sum, 
and the force $f_{\ell}(\uH)$ is the tight binding (QM) force of the infinite body.
We use the cutoff radii $\rcut$ to truncate the simulation domain and compute the force, hence, the approximation of \eqref{eq:error-indicator-fl} could be written as:
\begin{eqnarray}\label{eta_posteriori}
\eta_{\rcut}(\uH) := \sum_{\ell\in\Omega_{\rm c}} \log(2+|\ell|)\cdot\b|f^{\rcut}_{\ell}(\uH)\b|
\qquad{\rm with}\qquad
\Omega_{\rm c} = \L\bigcap\Big(\bigcup_{\ell\in\LQM\cup\LMM}B_{\rcut}(\ell)\Big) ,
\end{eqnarray}
where $f^{\rcut}_{\ell}(\uH) :=f_{\ell}^{B_{\rcut}(\ell)}(\uH)$ is the force computed from 
a finite system in the ball $B_{\rcut}(\ell)$,  defined by \eqref{Fl-El}.
Thanks to the locality result in Lemma \ref{lemma:regularity}, the error of this approximated error indicator (compared with \eqref{eq:error-indicator-fl}) decays exponentially fast to 0 as $\rcut$ increases.

\subsection{A sampling strategy for the evaluation of error indicator}
\label{sec:sample}

By Theorem \ref{theorem:upperbound}, the error indicator \eqref{eq:error-indicator-fl}  bounds the error of the approximate equilibrium state $\uH$ from above.
It may be more useful to have some local error indicator, 
in order to direct us to adjust the QM and MM regions automatically.
And furthermore, the computational cost is still very expensive since the evaluation of the error indicator \eqref{eta_posteriori} requires the computation of the QM forces $f_{\ell}(\uH)$
at all sites $\ell\in\Omega$, which is very expensive in real simulations.

Therefore, we propose an strategy to partition the simulation domain into local elements and construct local error indicator on each element.
A good partition can also help us compute the {\it a posteriori} error indicator efficiently.
We can sample one or a few sites in each element, and compute the force on the sampled sites to represent the error distribution in this element.

We will focus on the two dimensional systems in this paper. The three dimensional implementation is in principle similar but technically more involved and we will leave it to future work.

We decompose the simulation domain $\Omega_{\rm c}$ in \eqref{eta_posteriori} with a partition $\T:=\{T\}$, such that $ \Omega_{\rm c}=\cup_{T\in\T} $.
We can then approximate $\eta_{\rcut}(\uH)$ by
\begin{eqnarray}\label{eta_sample}
	\nonumber
	\eta_{\rcut}(\uH) &=& \sum_{T\in\T}\sum_{\ell\in T} \log(2+|\ell|)\cdot\b|f^{\rcut}_{\ell}(\uH)\b| 
	\\[1ex]
	&\approx& \sum_{T\in\T} w(T) \log(2+|\tilde{\ell}(T)|)\cdot\b|f_{\tilde{\ell}(T)}^{\rcut}(\uH)\b|
	~=:~ \eta_{\rcut}^{\T}(\uH),
\end{eqnarray}
where $\tilde{\ell}(T)$ denotes the repatom of $T$, and $w(T)$ gives the weight of the element $T\in\T$ 
(e.g., one can take $w(T)$ be the number of sites in $T$, or the relative area of $T$). 
We use 
\begin{eqnarray}\label{eta_local}
 \eta_{\rcut}^{T}(\uH) = w(T) \log(2+|\tilde{\ell}(T)|)\cdot\b|f_{\tilde{\ell}(T)}^{\rcut}(\uH)\b|
\end{eqnarray}
to denote the local error indicator on element $T$, which will provide us the information 
for the model adjustments in the adaptive algorithm. Multiple repatoms within one simplex $T$ is also possible and can be chosen by, e.g., Gauss-Lobatto quadrature rule.

\begin{remark}
The motivation behind the partition and sampling is that the force distribution is smooth in most of the area, 
for example, in the area away from the QM region (see e.g. Figure \ref{qmmmfrc} ).
Therefore, the sampling can keep the accracy of the error indicator while at the same time significantly reducing the computational cost.
\end{remark}

The choice of partition $\mathcal{T}$ is crucial for the accuracy and efficiency of the evaluation of the error indicator. 
In this paper, we focus more on local point defects, and partition the simulation domains in polar coordinates. 
The following partition strategy generates a graded mesh along the radius direction. The efficiency of this strategy is shown by our numerical experiments for some prototypical problems (see Section \ref{sec:numerics}).

For two dimensional quasi spherically symmetric defect configuration, for example, single point defect or microcrack, or multiple point defects, we have the following algorithm.

\vskip 0.2cm

\begin{algorithm}[H]
\caption{Graded mesh generation}
\label{alg:grademesh}
\begin{enumerate}
	\item
	Let $n_{\theta}\in\Z_+$ and $\tau=2\pi/M$.
	Set $0=\theta_0 < \theta_1 < \cdots < \theta_{n_{\theta}} = 2\pi$ with $\theta_{j+1} = \theta_j + \tau$.
	
	\item
	Let $n_r\in\Z_+$ and $n_r = n_{\rm QM} + n^1_{\rm MM} + n^2_{\rm MM} + n_{\rm FF}$.
	Let 
	\begin{displaymath}
	\begin{array}{ll}
	 h_1\geq\cdots\geq h_{n_{\rm QM}}>0 , & \text{ from defect core to QM/MM interface }\\
	 0<h_{n_{\rm QM}+1}<\cdots<h_{n_{\rm QM}+n^1_{\rm MM}}  , & \text{ from QM/MM interface to the coarsest $T\in \mathcal{T}$} \\
	 h_{n_{\rm QM}+n^1_{\rm MM}+1}>\cdots>h_{n_r+n^1_{\rm MM}+n^2_{\rm MM}} >0 , & \text{ from the coarsest $T\in\mathcal{T}$ to MM/FF interface}\\
	 0<h_{n_r-n_{\rm FF}+1}<\cdots<h_{n_r} , & \text{ far field}
	\end{array}
	\end{displaymath}
	such that
	$\displaystyle \sum_{k=1}^{n_{\rm QM}} h_k = \RQM$, 
	$\displaystyle \sum_{k=1}^{n_r-n_{\rm FF}} h_k = \RMM$
	and $\displaystyle \sum_{k=+1}^{n_r} = \RMM + \rcut$.
	
	Set $0=r_0<r_1<\cdots<r_{n_r}$ with $r_k = r_{k-1} + h_k$.
	\item
	Let $\T=\{T_{ij}\}$, $T_{ij} = (r_{i-1}, r_i]\times(\theta_{j-1}, \theta_j]$ in polar coordinate
	with $i = 1,\cdots, n_r$ and $j = 1,\cdots, n_{\theta}$.
	Let $\tilde{\ell}_{ij}\in T_{ij}$ be the site that is closest to the centre of $T_{ij}$, 
	and $w(T_{ij})$ be the number of atoms that lies in $T_{ij}$. 
\end{enumerate}
\end{algorithm}
 



\begin{remark}
In the numerical experiments in Section \ref{sec:numerics}, we use the following parameters in Algorithm \ref{alg:grademesh}, 
\begin{align*}
	 &h_1=\cdots= h_{n_{\rm QM}}=1 , \\
	 &h_{n_{\rm QM}+j} = \frac12(\frac{\sum_{k=1}^{n_{\rm QM}+j-1}h_k}{n_{\rm QM}})^{1.5}+\frac12(\frac{\sum_{k=1}^{n_{\rm QM}+j}h_k}{n_{\rm QM}})^{1.5} ,
	 \quad{\rm for}~~1\leq j\leq n^1_{\rm MM},  \\
	&h_{n_{\rm QM}+n^1_{\rm MM}+j} = h_{n_{\rm QM}+n^1_{\rm MM}-j} ,
	\quad{\rm for}~~1\leq j\leq n^2_{\rm MM},\\
	&h_{n_r-n_{\rm FF}+j}=h_{n_{\rm QM}+j} ,
	\quad{\rm for}~~1\leq j\leq n_{\rm FF}.
\end{align*}
\end{remark}

\begin{remark}
	For more general defect configurations, one may generate adaptive mesh according to the error indicator, for example, starting from a coarse partition. The technical details will appear in our forthcoming paper.
\end{remark}

\section{Adaptive QM/MM algorithms}
\label{sec:adaptive}
\setcounter{equation}{0}
\setcounter{figure}{0}

In this section, we design an adaptive QM/MM algorithm for crystalline defects based on the {\it a posteriori} error indicator \eqref{eta_local}.
The basic idea of the adaptive method  is to repeat the following procedure before reaching the required accuracy:
$$
\mbox{Solve}~\rightarrow~\mbox{Estimate}~\rightarrow~
\mbox{Mark}~\rightarrow~\mbox{Refine}.
$$

Given a partition $\LQM$ and $\LMM$, the ``Solve" step computes the approximate equilibrium state $\uH$ by solving \eqref{problem-e-mix}.
The ``Estimate" step computes the {\it a posteriori} error indicators \eqref{eta_sample} and \eqref{eta_local}.
The ``Mark" step uses some adaptation strategy to choose sets $\Mrefine$ for model refinement, and here it refers to the change of $\LQM$/$\LMM$ interface and $\LMM$/$\LFF$ interface.
We choose the following  D\"{o}rfler strategy, which is a widely used marking strategy to enforce error reduction. 

\begin{algorithm}[H]
\caption{D\"{o}rfler Strategy.}
\label{alg:dorfler}
\quad Prescibe $0<\tau<1$.
\begin{enumerate}
	\item 
	Choose the minimum set $\Mrefine\subset\L$ such that the following D\"{o}rfler properties are satisfied
	\begin{eqnarray}\label{dorfler-strategy}
	\sum_{T\subset \Mrefine}\eta^T_{\rcut}(\uH)  \geq \tau \sum_{T\in \mathcal{T}} \eta^T_{\rcut}(\uH) .
	\end{eqnarray}
	\item 
	Mark all the sites in $\Mrefine$ (for refinement).
\end{enumerate}
\end{algorithm}

\def\dist{\textrm{dist}}

To select the minimum set $\Mrefine$, 
we first sort $\eta_T(\uH)$ in descending order for $T\in\mathcal{T}$, 
then compute the partial sums of the sorted sequence until \eqref{dorfler-strategy} is satisfied.
In the QM/MM coupling, we should further decompose the marked set into two nonintersecting subsets
\begin{eqnarray*}
\Mrefine = \Mrefine^{\rm QM} \cup \Mrefine^{\rm MM}, \qquad  \Mrefine^{\rm QM} \cap \Mrefine^{\rm MM} = \emptyset.
\end{eqnarray*}
We will see from the numerical tests (see Figure \ref{fig:side:a} and \ref{fig:qmmmgd} (b)) 
that the force distribution are mainly concentrated near the $\LQM/ \LMM$ interface and $\LMM/\LFF$ interface.
Therefore, the simplest way to determine this decomposition is by the position of the atomic site $\ell\in\Mrefine$: 
if $\dist (\ell, \LQM) < \dist(\ell, \LFF)$ , then $\ell$ goes into $\Mrefine^{\rm QM}$; otherwise, $\ell$ goes into $\Mrefine^{\rm MM}$.
%
Then, with the D\"{o}rfler adaptation strategy, $\Mrefine$ consequently lies in two (separated) regions around two interfaces, which can easily be decomposed into the QM and the MM parts.
%
\begin{remark}
\label{rem:generalalg}
The decomposition of $\Mrefine$ into QM and MM parts requires the assumption that the errors are concentrated around $\LQM/ \LMM$ and $\LMM/\LFF$ interfaces. This seems to be ad hoc, which may not work for general defect configurations. But the main purpose of this paper is to develop an analytical framework for adaptive QM/MM computation, and justify it numerically by some prototypical examples, such as the single vacancy and two separated vacancies in Section \ref{sec:numerics} where such an assumption holds. A more general numerical approach may need to combine ideas such as stress based error indicator from adaptive atomistic/continuum coupling method \cite{Wang:2017,Liao2018}, and will be investigated in our future work. 
\end{remark}

\begin{remark}
	One can also use the so-called maximum marking strategy for the ``Mark" step.
	The maximum strategy chooses a $T_{\rm max}\in \mathcal{T}$, such that  
	\begin{eqnarray*}\label{maximum-strategy}
		\eta_{\rcut}^{T_{\rm max}}(\uH) = \max_{T\in\mathcal{T}}\eta_{\rcut}^T(\uH) .
	\end{eqnarray*}
	and 	$\Mrefine$ contains all sites in $T_{\rm max}$. In this paper we will stick to the D\"{o}rfler strategy, which behaves more efficiently in all our numerical examples. 
\end{remark}

Once the marked sets $\Mrefine^{\rm QM}$ and $\Mrefine^{\rm MM}$ are determined, 
the ``Refine" step adjusts the domain decomposition accordingly for the next ``Solve" step. 
Note that usually more atomic sites than that in the marked sites $\Mrefine$ are refined, in order to keep the QM and MM regions regular.

The adaptive QM/MM algorithm is given as follows.

\begin{algorithm}[H]
\caption{Adaptive QM/MM algorithm}
	\label{alg:main}	
	\quad
	\begin{enumerate}		
		\item
		Prescribe $\varepsilon_{\rm tol}>0$, $\mathrm{r}\in(0,1)$, $N_{\rm QM}^{\max}$, $N_{\rm MM}^{\max}$ and $\Rbuf$.
		Initialize $\LQM$ and $\LMM$.
		Construct $\Lbuf\subset\LMM$ such that \eqref{buf} is satisfied.
		\item
		If $\#\LQM>N_{\rm QM}^{\max}$ or $\#\LMM>N_{\rm MM}^{\max}$,  \textbf{STOP}; 
		otherwise solve \eqref{problem-e-mix} to obtain $\bar{u}^{\rm H}$. 
		\item
		Compute the error indicator $\eta_{\rcut}^{\T}(\uH)$ in \eqref{eta_sample} and $\eta_{\rcut}^{T}(\uH)$ in \eqref{eta_local} for each $T\in\mathcal{T}$.
		If  $\eta_{\rcut}(\uH)<\varepsilon_{\rm tol}$, \textbf{STOP}; otherwise, go to Step 4.
		\item
		Use D\"{o}rfler Strategy to construct $\Mrefine$, and  decompose $\Mrefine$ into $\Mrefine^{\rm QM}$ and $\Mrefine^{\rm MM}$. 
		\item
		Construct new $\LQM$, $\LMM$ and $\Lbuf$ such that 
		$\LQM\supset\Mrefine^{\rm QM}$, $\LMM\supset\Mrefine^{\rm MM}$ and \eqref{buf} is satisfied, go to Step 2.
	\end{enumerate}
\end{algorithm}

\section{Numerical experiments}
\label{sec:numerics}
\setcounter{equation}{0}
\setcounter{figure}{0}

In this section, we will complement our theoretical analysis with numerical experiments. 
We consider two-dimensional triangle lattice $\L^{\textrm{hom}}:=\mA\Z^{2}$ with
embedded local point defects, where
\begin{equation}
\mA = \mymat{1 & \cos(\pi/3) \\ 0 & \sin(\pi/3)}.
\label{eq:Atrilattice}
\end{equation}
For the tight-binding model, we use a simple toy model with the Hamiltonian given in \eqref{tb-H-elements}, where the onsite term is $h_{\rm ons} = 0$, and the hopping term is given by the Morse potential
\begin{eqnarray*}
h_{\rm hop} (r) =  e^{-4(r-1)} \quad{\rm for}~~r>0.
\end{eqnarray*}

\noindent
{\bf Example 1.} (Single vacancy)
Consider a single vacancy located at the origin with $\L = \Lhom\backslash\{\pmb 0\}$.

We first use pure QM (tight binding) calculations to verify the convergence and reliability of our {\it a posteriori} error indicator,
where the QM subsystems is embedded directly in a bulk environment (without any MM subsystems). The geometry of the partition of QM and far field regions are shown in Figure \ref{fig:qmgeom}, where red atoms are simulated by QM model surrounded by far field atoms colored in green. 
We still denote the approximate equilibrium solution by $\uH$.

We observe from Figure \ref{fig:qmape} that our error indicator decays in the same rate as 
$\|\bar{u} - \uH||_{\UsH}$ while the QM region increases. This not only supports the reliability (upper bound estimate) in Theorem \ref{theorem:upperbound}, but also shows the efficiency (lower bound estimate) of our error indicator.
We also compare the decay of error indicators with different cutoff $\rcut$ in Figure \ref{fig:side:b}.
It is observed that the choice of $\rcut$ does affect the reliability of our error indicator.
From our numerical simulations, we see that $\rcut=5$ or 6 is good enough
and will be used throughout the following numerical experiments. 

We present the force distribution (with respect to the radii) in Figure \ref{fig:side:a}. 
Instead of using $f_{\ell}^{\rcut}(\uH)$ in \eqref{eta_posteriori}, we plot the force $f_{\ell}^{B_R(0)}(\uH)$ computed on a very large simulation domain with $R\gg\RQM+\rcut$.
Note that $f_{\ell}^{B_R(0)}(\uH)$ makes very accurate approximation of the true QM force $f_{\ell}(\uH)$ in the thermodynamic limit.
The figure shows that the force are more concentrated around the interface between QM and far field regions.

\begin{figure}[!htb]
	\centering 
	\subfigure[Partition of the QM and far field region.]{
		\label{fig:qmgeom}
		\includegraphics[height=6cm]{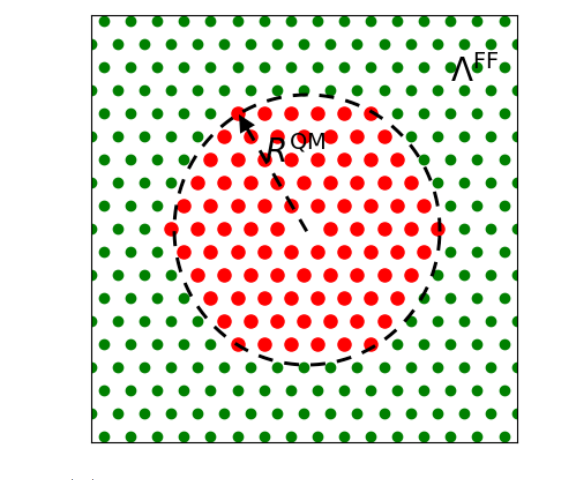}}
	\hspace{0.2cm} 
	\subfigure[Verification of the {\it a posteriori} error indicator.]{
		\label{fig:qmape} 
		\includegraphics[height=6cm]{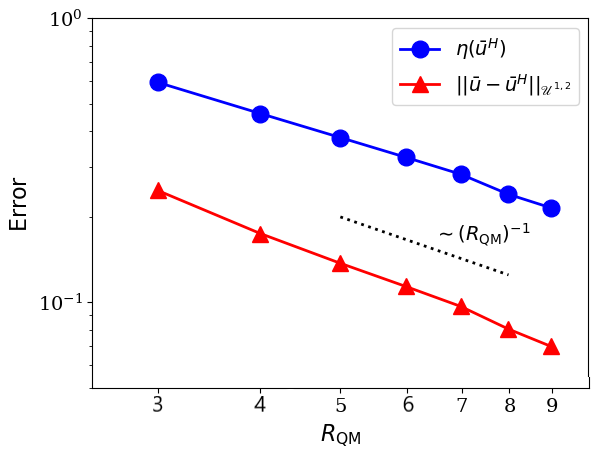}}
	\caption{Comparison of the {\it a posteriori} error indicator and $\|u - \bar{u}\|_{\UsH}$ with pure QM calculations.}
	\label{fig:pureQM}
\end{figure} 

\begin{figure}[!htb]
	\begin{minipage}[t]{0.5\linewidth}
		\centering
		\includegraphics[scale=0.5]{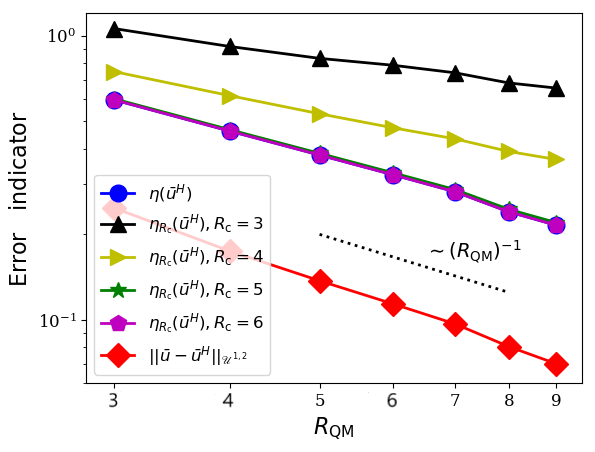}
		\caption{Error indicators with different $\rcut$. 
		}
		\label{fig:side:b}
	\end{minipage}
	\begin{minipage}[t]{0.5\linewidth}
		\centering
		\includegraphics[scale=0.5]{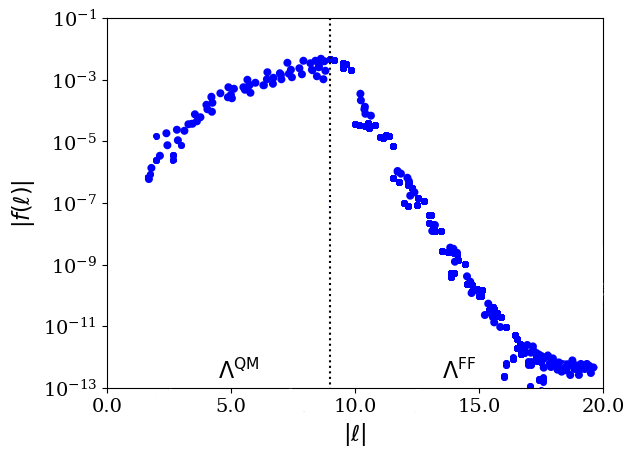}
		\caption{Force distribution along the radii with a pure QM calculation.}
		\label{fig:side:a}
	\end{minipage}
\end{figure}

We show the QM/MM domain decomposition in Figure \ref{qmmmgeom},  and the corresponding force distribution (with fixed partition)
with respect to the radius direction in Figure \ref{qmmmfrc}. Similar to Figure \ref{fig:side:a}, we plot the force $f_{\ell}^{B_R(0)}$ on a very large system with $R\gg\RMM+\rcut$.
We observe that forces are mainly concentrated around the $\LQM/ \LMM$ interface and the $\LMM/\LFF$ interface.
This motivates the adaptive Algorithm \ref{alg:main}, which will assigns more sample points near those two interfaces.
We show the sample points in Figure \ref{gmsamp} and the elapsed time for the evaluation of the error indicator with/without sampling algorithm in Figure \ref{fig:surf_single}.
It is clear that the evaluation time for the error indicator with sampling \eqref{eta_sample} is significantly reduced compared with that without sampling \eqref{eta_posteriori}. 
Furthermore, the sampling algorithm does not affect the accuracy of the error indicators (see Figure \ref{singlegooderr}).

\begin{figure}[!htb]
	\centering 
	\subfigure[Partition of the domain.]{
		\label{qmmmgeom}
		\includegraphics[scale=0.55]{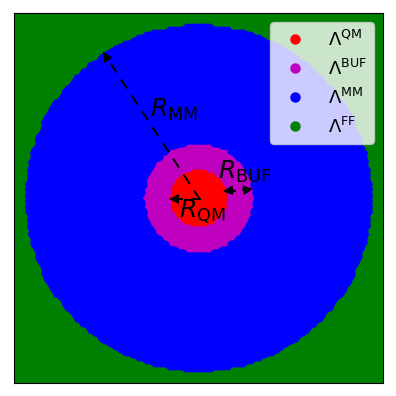}}
	\hspace{0.2cm} 
	\subfigure[Force distribution along the radii with a QM/MM coupling.]{
		\label{qmmmfrc} 
		\includegraphics[scale=0.5]{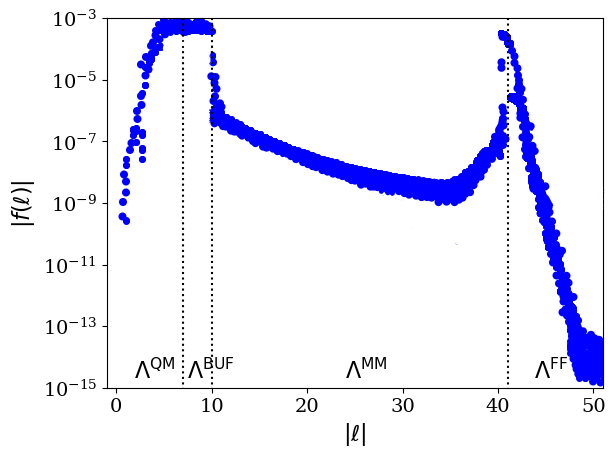}}
	\caption{Domain decomposition in the QM/MM coupling scheme, and force distribution along the radii.}
	\label{fig:qmmmgd}
\end{figure} 

\begin{figure}[!htb]
	\centering 
	\subfigure[Graded mesh sampling.]{
		\label{gmsamp}
		\includegraphics[scale=0.58]{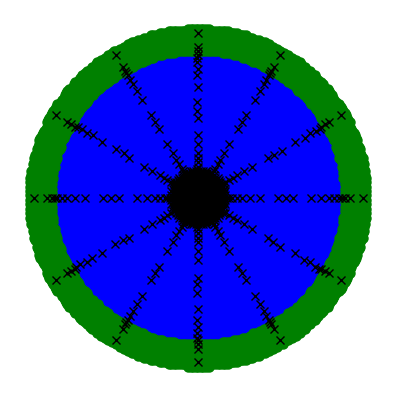}}
	\subfigure[Computing time for the {\it a posteriori} error indicator.]{
		\label{fig:surf_single}
		\includegraphics[scale=0.51]{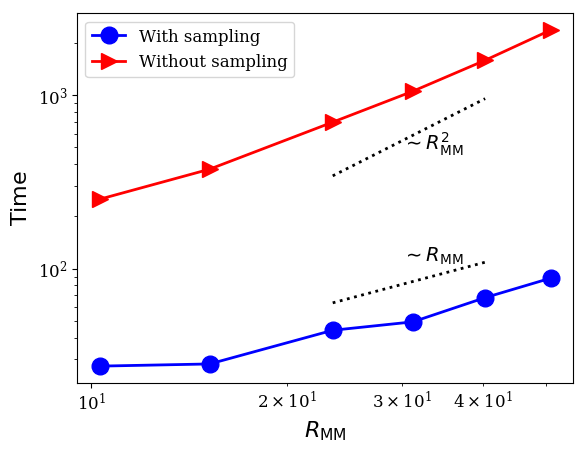}}	        
	\caption{Sampling points for the {\it a posteriori} error indicator, and the scaling of computational time.}
	\label{fig:samp}
\end{figure}

We then perform the adaptive algorithm (Algorithm \ref{alg:main}) to compute the single vacancy example. In each ``Solve" step, the computational cost is proportional to $N_{\rm QM}^3+N_{\rm MM}$,
as the cost to solve the tight binding model scales cubically and the cost to solve the MM model scales linearly with respect to the number of atoms.
The decay curves for the errors of QM/MM solutions and the {\it a posteriori} error indicators are shown in Figure \ref{singlegooderr}, as a function of $N_{\rm QM}^3+N_{\rm MM}$.
The relation between $N_{\rm QM}$ and $N_{\rm MM}$ during the adaptation process is shown in Figure \ref{singlegoodpath}, from which we observe that our adaptive algorithm can achieve optimal computational complexity.

\begin{figure}[!htb]
	\centering 
	\subfigure[Decay of the error in the adaptive algorithm.]{
		\label{singlegooderr}
		\includegraphics[scale=0.5]{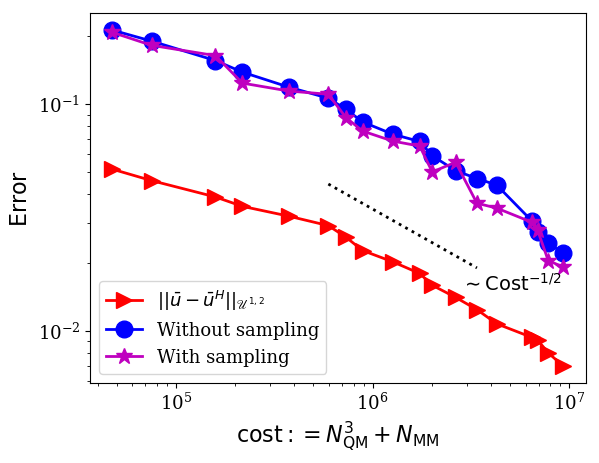}}
	\hspace{0.2cm} 
	\subfigure[Relation between $\RQM$ and $\RMM$ in the adaptive algorithm.]{
		\label{singlegoodpath} 
		\includegraphics[scale=0.5]{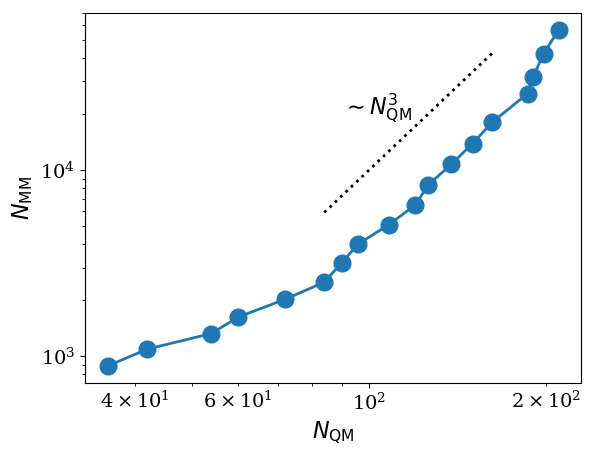}}
	\caption{Convergence of the adaptive algorithm, and the scaling of QM and MM radius.}
	\label{singlegood}
\end{figure}

\noindent
{\bf Example 2.} (Two separated vacancies)
Consider two vacancies that are away from each other (see Figure \ref{tsvgeom}).
%
Since the system has quasi-spherical symmetry away from the defects, we are still able to apply our graded mesh algorithm. We show one QM/MM partition in Figure \ref{tsvgeom} and the distribution of the sample points in Figure \ref{tsvape}, which are selected with respect to each vacancy core. 
Here we adapt the graded mesh generation Algorithm \ref{alg:grademesh} such that the sampling points in the left half plane are generated by the vacancy on the left, and the sampling points in the right half plane are generated by the vacancy on the right. See Remark \ref{rem:generalalg} for discussions on a more general approach.


In our adaptive simulations, the initial geometry contains two isolated QM regions.
The adaptive algorithm adjust the QM and MM regions automaticlly according to the error indicators.
We show the evolution of the QM/MM partitions during the adaptation process in Figure \ref{fig:TSVsteps},
and observe the merge and growth of QM subsystems as $N_{\rm QM}$ increases. 
We plot the the error indicators and true approximation errors in Figure \ref{tsvgooderr}, 
which shows the accuracy of our adaptive algorithm and the efficiency of the sampling techniques.
The relation between $N_{\rm QM}$ and $N_{\rm MM}$ is shown in Figure \ref{tsvgoodpath},  which implies that our adaptive algorithm can give optimal computational scaling.

\begin{figure}[htb]
	\centering 
	\subfigure[Geometry of two separated vacancies and QM/MM decompositions.]{
		\label{tsvgeom}
		\includegraphics[height=6cm]{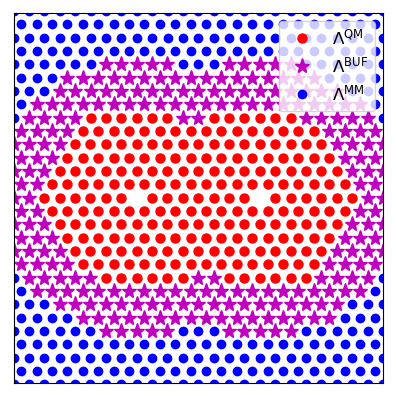}}
	\hspace{0.2cm} 
	\subfigure[Sample points for two separated vacancies.]{
		\label{tsvape} 
		\includegraphics[height=6cm]{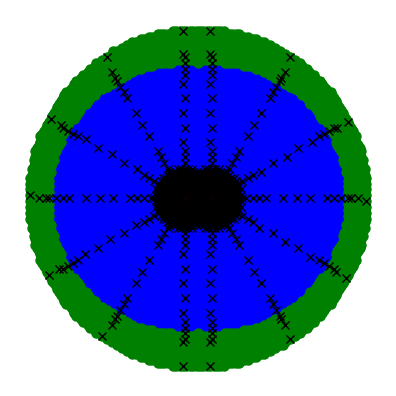}}
	\label{tsv2}
	\caption{Geometry of the QM/MM decompositions and corresponding sample points.}
\end{figure} 

\begin{figure}[htb]
	\centering 
	\includegraphics[height=4cm]{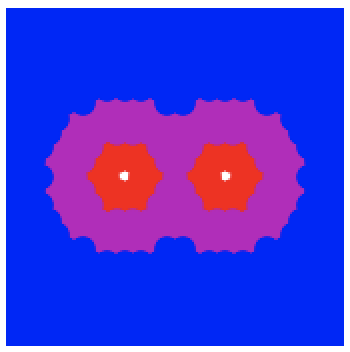}
	\includegraphics[height=4cm]{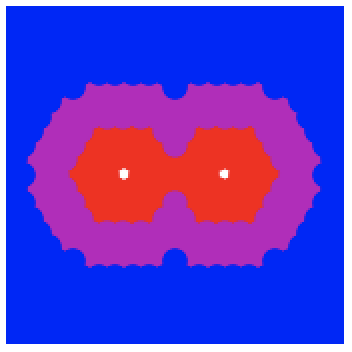}
	\includegraphics[height=4cm]{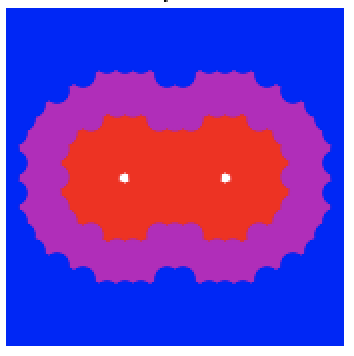}
	\includegraphics[height=4cm]{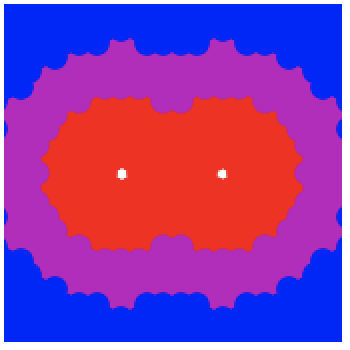}
	\caption{Evolution of QM and MM partition in the adaptation process. $N_{\rm QM}$, the numbers of atomic sites in the QM region are  36,  71,  95,  and  133 from left to right.}
	\label{fig:TSVsteps}
\end{figure} 

\begin{figure}[htb]
	\centering 
	\subfigure[Convergence curves of the error indicators and approximation erros.]{
		\label{tsvgooderr}
		\includegraphics[scale=0.53]{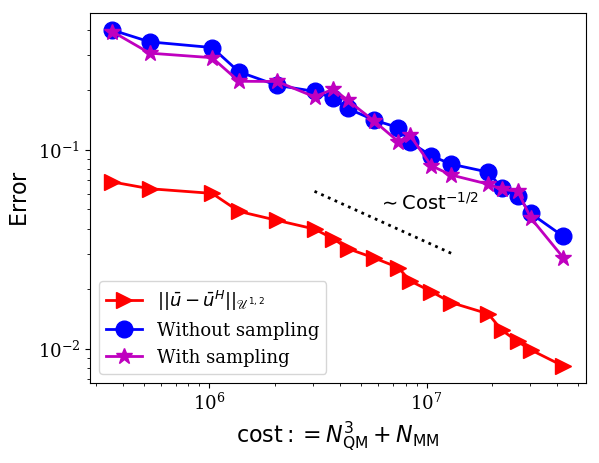}}
	\hspace{0.2cm} 
	\subfigure[$N_{\rm QM}$ vs $N_{\rm MM}$]{
		\label{tsvgoodpath} 
		\includegraphics[scale=0.53]{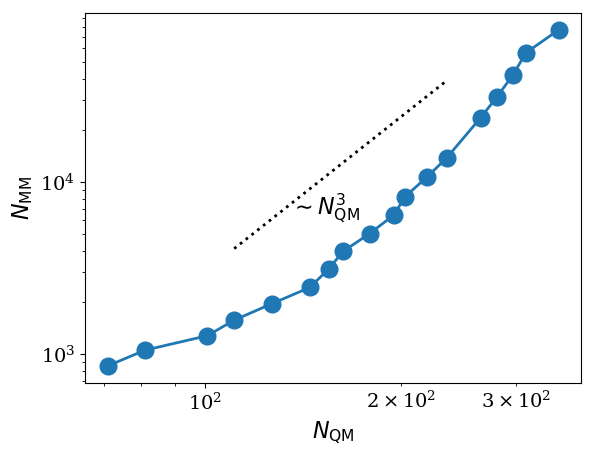}}
	\caption{Convergence of the adaptive algorithm, and the scaling of QM and MM radius.}
	\label{tsvgood}
\end{figure}

\section{Conclusions}
\label{sec:conclusion}
\setcounter{equation}{0}

In this paper, we provide an {\it a posteriori} error indicator for QM/MM coupling approximations, and design an adaptive algorithm for crystalline solids with embedded defects.
The error indicator not only gives an upper bound for the approximation error of the geometry equilibration, but also allows to adjust the QM/MM decomposition on the fly.
Moreover, the error indicator can be computed efficiently with a sampling algorithm.
We conclude that,
(a) more flexible sampling methods are required to compute the error indicator for more general defected systems, 
(b) our method is potentially more efficient and important for dynamic problems (with moving defects), where a coarsening process should be applied. 
These issues will be investigated in our future work.

\subsubsection*{Acknowledgements}

We are grateful to Christoph Ortner and Julian Braun from University of Warwick for stimulating discussions regarding this work.

\appendix
\renewcommand\thesection{\appendixname~\Alph{section}}

\small
\bibliographystyle{plain}
\bibliography{posterioriQMMM.bib}

\end{document}

%% file: notation.tex
\newcommand{\mymat}[1]{\left[ \begin{matrix} #1 \end{matrix} \right]}

\def\XXint#1#2#3{{\setbox0=\hbox{$#1{#2#3}{\int}$ }
\vcenter{\hbox{$#2#3$ }}\kern-.6\wd0}}

\def\b{\big}

\def\R{\mathbb{R}}

\def\Z{\mathbb{Z}}

\def\DD{{D}'}

\def\<{\langle}
\def\>{\rangle}

\def\mA{{\sf A}}

\def\del{\delta}

\def\L{\Lambda}

\def\Usz{\Us_0}

\def\E{\mathscr{E}}

\def\L{\Lambda}

\def\rcut{r_{\rm cut}}
\def\Rg{\mathscr{R}}

\def\T{\mathcal{T}}
